\documentclass[12pt]{amsart}

\usepackage{amsmath,amssymb,amsthm,verbatim}
\usepackage{hyperref}
\usepackage{color}

\newtheorem{theorem}[subsection]{Theorem}
\newtheorem{lemma}[subsection]{Lemma}

\theoremstyle{definition}

\newtheorem{example}[subsection]{Example}

\newcommand{\haus}{\mathcal{H}}
\newcommand{\leb}{\mathcal{L}}
\newcommand{\spt}{\mathrm{spt}}

\newcommand{\graph}{\mathrm{graph}}

\newcommand{\eps}{\epsilon}

\newcommand{\R}{\mathbb{R}}

\newcommand{\del}{\partial}

\newcommand{\cI}{\mathcal{I}}

\begin{document}

\title[Bernstein Type Theorem]{A Bernstein-type theorem for minimal graphs over convex domains}
\author[Nick Edelen]{Nick Edelen}
\address{Department of Mathematics\\
University of Notre Dame\\
Notre Dame, IN 46556, USA} \email{nedelen@nd.edu}
\author[Zhehui Wang]{Zhehui Wang}
\address{Beijing International Center for Mathematical Research\\
Peking University\\
Beijing, 100871, China}
\email{wangzhehui@pku.edu.cn}
\maketitle

\begin{abstract}
Given any $n \geq 2$, we show that if $\Omega \subsetneq \R^n$ is an open convex domain (e.g. a half-space), and $u : \Omega \to \R$ is a solution to the minimal surface equation which agrees with a linear function on $\del \Omega$, then $u$ must itself be linear.
\end{abstract}

\section{Introduction}

The classical Bernstein theorem asserts that any entire minimal graph over $\R^n$ with $2\leq n\leq 7$ is a hyperplane. This was originally proven when $n=2$ by Bernstein \cite{Bern1015},  $n=3$ by De Giorgi \cite{DeGiorgi1965},  $n=4$ by Almgren \cite{Alm1966}, and $n=5, 6, 7$ by Simons \cite{Simons1968}.  For $n\geq 8$ there were counter examples constructed by Bombieri, De Giorgi and Giusti \cite{BGG1969}.

Here we prove a Bernstein-type theorem for minimal graphs with Dirichlet boundary conditions in a convex subdomain of $\R^n$.
\begin{theorem}\label{thm:main1}
Let $\Omega \subsetneq \R^n$ be an open, convex subset of $\R^n$.  Suppose $u \in C^2(\Omega) \cap C^0(\overline{\Omega})$ solves
\[
\sum_i D_i \left( \frac{ D_i u}{\sqrt{1+|Du|^2}} \right) = 0 \text{ in } \Omega, \quad u = l \text{ on } \del\Omega
\]
for some linear function $l : \R^n \to \R$.  Then $u$ is linear, and if $\Omega$ is not a half-space, then in fact $u = l$.
\end{theorem}

We emphasize that unlike the entire case, we require no dimensional restriction.  This is due to the very rigid nature of area-minimizing hypersurfaces-with-boundary contained in a convex cylinder $\overline{\Omega} \times \R$, and our Theorem \ref{thm:main1} is a consequence of a more general Bernstein-type theorem for these types of surfaces.  We state this here.  See the following section for notation.
\begin{theorem}\label{thm:main2}
Let $\Omega \subsetneq \R^n$ be an open, convex subset of $\R^n$, $l : \R^n \to \R$ a linear function, and $T \in \cI_n(\R^{n+1})$ a mass-minimizing integral $n$-current in $\R^{n+1}$.  Suppose $T$ has the form $T = \del [E] - [G]$ for $E \subset \R^{n+1}$ satisfying
\[
E \setminus (\overline{\Omega} \times \R) = \{ x_{n+1} < l(x_1, \ldots, x_n) \} \setminus (\overline{\Omega} \times \R)
\]
and $[G] = [\graph(l|_{\R^n \setminus \overline{\Omega}})]$ endowed with the upwards orientation.  Then:
\begin{enumerate}
\item $T = [H]$ for some half-hyperplane $H$, if $\Omega$ is a half-space;
\item $T = [\graph(l|_\Omega)]$, if $\Omega$ is not a half-space.
\end{enumerate}
\end{theorem}

We also remark that even when $n \leq 7$ and $\Omega$ is a half-space, then for general $l$ one cannot simply reflect to reduce to the standard Bernstein problem, as one may not obtain a graph after reflection.  (Of course, if $u$ has $0$ Dirichlet or Neumann data, then a reflection argument like this would work).

Many others have studied the behavior of minimal graphs over unbounded domains in $\R^n$.  In exterior domains $\R^n \setminus \Omega$ ($\Omega$ bounded and $2 \leq n \leq 7$), \cite{Bers}, \cite{SimAsym} proved that the gradient of a minimal graph $u$ is bounded and has a limit as $|x| \to \infty$.  When $n \geq 8$, $Du$ may not be bounded, but \cite{SimEnt} showed precise cylindrical asymptotics;  \cite{Weit}, \cite{SJ2002} proved related results for domains in $\R^2$.

When $\Omega \subset \R^2$ is contained in a wedge of opening angle $< \pi$, then \cite{ER1989} showed that any minimal graph over $\Omega$ with zero Dirichlet data must be $0$.  In a similar vein, when $\Omega \subset \R^2$ is a wedge with opening angle $\neq \pi, 2\pi$, then \cite{HSwedge1999} showed any minimal graph with zero Neumann data away from the cone point must be constant (see also \cite{LCC1997}).  In \cite{JWZ2019} the second named author and his collaborators obtained Liouville type theorems for minimal graphs over half-spaces (for any $n$) with linear Dirichlet boundary value or constant Neumann boundary value.

\vspace{3mm}

If $\spt T \not\subset \overline{\Omega} \times \R$, or one relaxes the condition $T = \del [E] - [G]$, then Theorem \ref{thm:main2} can fail.
\begin{example}
Half of Enneper's surface, and the half-helicoid, are both area-minimizing hypersurfaces in $\R^3$ which bound a line (\cite{Wh96}; c.f. \cite{Ed}).  Neither is contained in a half-space.  The tangent cone at infinity to half-Enneper's surface is a plane with multiplicity $2$ on one side, and multiplicity $1$ on the other.  The half-helicoid does not even have quadratic area growth.  Of course a plane with multiplicity $j$ on one side, $j+1$ on the other, is area-minimizing, and contained in a half-space, but does not admit a decomposition of the form $\del [E] - [G]$ unless $j = 0$.

The surface $T = \{ x_2 = 0, x_3 \geq 0 \} \cup \{ x_2 = 1, x_3 \leq 0 \}$ endowed with the appropriate orientation is area-minimizing in $\R^3$ (the constant vector field $e_2$ is a calibration), contained in a slab $\R \times [0, 1] \times \R$, and has linear boundary $[\R\times \{(0, 0)\}] - [\R \times \{(1, 0)\}]$, but is not contained in hyperplane.  In this case the surface fails to admit the appropriate boundary structure $T = \del [E] - [G]$.

On the other hand, we don't know any example of a non-planar area-minimizing surface with compact boundary contained in a hyperplane -- it is possible one might be able to weaken the hypotheses on $T$ when $\Omega$ is not a half-space or a slab.
\end{example}

We first prove a technical Lemma \ref{lem:graph-minz} which allows us to reduce Theorem \ref{thm:main1} to Theorem \ref{thm:main2}.  We use the convexity assumption most strongly here.  We prove Theorem \ref{thm:main2} using a barrier argument, which is slightly different depending on whether $\Omega$ is a half-space, a slab, or neither of those two (in which case $\Omega$ is contained in a convex cone, itself being a proper subset of some half-space).  The basic idea is to rotatate a half-plane around the linear boundary of $T$ until it touches $T$, and then use the strong-maximum principle for minimal surfaces \cite{SW}.  This is easiest in the third case, and in fact here we do not even require convexity of $\Omega$, only that $\Omega$ is contained in some cone, which itself contains no half-space.  In the second case (when $\Omega$ is a slab), we need to additionally use the structure $T = \del [E]$, to rule out the possibility that $T$ looks like a vertical plane.  The first case (when $\Omega$ is a half-space) is in some sense the least rigid, and we apply the barrier argument to a tangent cone of $T$ rather than $T$ itself, to show that any mass-minimizing multplicity-one hypercone contained in a half-space and with linear boundary must itself be a multiplicity-one half-plane.  This classification result for minimizing cones-with-boundary is actually a corollary of the remarkable theorem due to \cite{HS79}, which classifies \emph{any} mass-minimizing hypercone with linear boundary as linear, but since our setting is vastly simpler we provide a much shorter, largely self-contained proof.

\section{Preliminaries}

We first outline our notation.  Unless otherwise stated we allow $n\geq 1$.  Let us write $e_1, \ldots, e_{n+1}$ for the standard basis of $\R^{n+1}$.  For $k < n+1$, we identify $\R^k \equiv \R^k \times \{0^{n+1-k}\} \subset \R^{n+1}$.  Given $v \in \R^{n+1}$, write $\R v = \{ r v : r \in \R \}$.  Given $U \subset \R^{n+1}$, write $U + \R v = \{ u + r v : u \in U, r \in \R \}$.  Let $\eta_{x, r}(y) = (y - x)/r$ be the translation/dilation map.

If $u : \Omega \subset \R^n \to \R$, write $\graph(u) = \{ (x', u(x')) : x' \in \Omega \}$ for the graph of $u$ in $\R^{n+1}$.  We shall always assume $\graph(u)$ is endowed with the ``upwards'' orientation (i.e. so that the normal $\nu$ satisfies $\nu \cdot e_{n+1} > 0$).  Given a set $U \subset \R^{n+1}$, write $\overline{U}$ for the closure of $U$, and write $d(x, U) = \inf \{ |x - y| : y \in U \}$ for usual Euclidean distance.  For $a \in \R^{n+1}$, $B_r(a)$ always denotes the open Euclidean $r$-ball centered at $a$, and more generally $B_r(U) = \{ x : d(x, U) < r \}$ is the open $r$-tubular neighborhood of $U$.  Write $\haus^k$ for the $k$-dimensional Hausdorff measure, and $\omega_k = \haus^k(B_1^k)$ for the $k$-volume of the $k$-dimensional unit ball.

Recall that $T$ is an integer-multiplicity rectifiable $k$-current in an open set $U$ if there is a countably $k$-rectifiable set $M_T \subset U$, a $\haus^k \llcorner M_T$-measurable simple unit $k$-vector $\xi_T$ orienting $T_x M_T$ for $\haus^k$-a.e. $x$, and a $\haus^k \llcorner M_T$-measurable positive integer-valued function $\theta_T$, so that
\[
T(\omega) = \int_{M_T} <\omega, \xi_T>\theta_T d\haus^k
\]
for every smooth, compactly-supported $k$-form in $U$.  We write $||T|| \equiv \haus^k \llcorner \theta_T \llcorner M_T$ for the mass measure of a current.  Given a Lipschitz map $f : \R^{n+1} \to \R^{n+1}$ which is proper when restricted to $\spt T$, we write $f_\sharp T$ for the pushforward.

We denote by $\cI_k(U)$ the set of \emph{integral} $k$-currents in $U$: that is, integer-multiplicity rectifiable $k$-currents $T$ in $U$ with the property that both $||T||$ and $||\del T||$ are Radon measures in $U$.  We say $T \in \cI_k(U)$ is mass-minimizing in an open set $U' \subset U$ if, for every $W \subset\subset U'$ and every $S \in \cI_k(U')$ satisfying $\del S = 0$, $\spt S \subset W$, we have
\[
||T||(W) \leq ||T+S||(W).
\]

If $L$ is a Lipschitz submanifold with a fixed (for the duration of this paper) $\haus^k \llcorner L$-measurable choice of orientation, then we will write $[L]$ for the multiplicity-one rectifiable $k$-current induced by integration.  If $E$ is a subdomain of $\R^{n+1}$, we will always orient $E$ with the constant, positive orientation of $\R^{n+1}$.  $E$ is called a set of locally-finite perimeter in $U$ if $[E] \in \cI_{n+1}(U)$.

We will use the following basic facts about convex domains.
\begin{lemma}\label{lem:convex}
If $\Omega$ is an open convex subdomain of $\R^{n}$, then:
\begin{enumerate}
\item $\del \Omega$ is locally-Lipschitz, and hence $\Omega$ is a set of locally-finite perimeter;
\item the nearest point projection $p_{\Omega} : \R^{n} \to \overline{\Omega}$ is $1$-Lipschitz;
\item for any ball $B_r(x)$, we have $\haus^{n-1}(\overline{B_r(x)} \cap \del\Omega) \leq \haus^{n-1}(\del B_r(x))$. \qedhere
\end{enumerate}
\end{lemma}
Point c) follows from points a) and b) by showing that $p_{\Omega} : \del B_r(x) \to \overline{B_r(x)} \cap \del\Omega$ is onto.  We leave the rest of the proof to the reader.

Our main technical lemma, which allows us to reduce Theorem \ref{thm:main1} to a problem about mass-minimizing boundaries, is the following.
\begin{lemma}\label{lem:graph-minz}
Let $\Omega$ be an open convex domain in $\R^n$, and $l : \R^n \to \R$ a locally-Lipschitz function.  Take $u \in C^2(\Omega) \cap C^0(\overline{\Omega})$ a solution to
\[
\sum_i D_i \left( \frac{D_i u}{\sqrt{1+|Du|^2}} \right)  = 0 \text{ on } \Omega, \quad u = l \text{ on } \del\Omega,
\]
Define
\begin{align*}
E &= \{ (x', x_{n+1}) \in \Omega \times \R : x_{n+1} < u(x') \} \\
 &\quad \cup \{ (x', x_{n+1}) \in (\R^n \setminus \Omega) \times \R : x_{n+1} < l(x') \},
\end{align*}
and $T = [\graph(u|_{\Omega})]$, $[G] = [\graph(l|_{\R^n \setminus \overline{\Omega}})]$, both oriented with the upwards unit normal.

Then:
\begin{enumerate}
\item $E$ is a set of locally-finite perimeter in $\R^{n+1}$, and $\del [E] = T + [G]$, $||\del [E]|| = ||T|| + ||[G]||$.
\item $T$ is an integral mass-minimizing current in $\R^{n+1}$.
\item $||T||(B_r(0)) \leq c(n) r^n$ for all $r > 0$.
\end{enumerate}
\end{lemma}

\begin{proof}
We first observe that by the usual calibration argument, $T$ is mass-minimizing in $\Omega \times \R$.  In other words, given any $S \in \cI_n(\Omega \times \R)$ with $\del S = 0$, and $\spt S \subset W \subset\subset \Omega \times \R$, then
\begin{equation}\label{eqn:graph-minz}
||T||(W) \leq ||T + S||(W).
\end{equation}

Let us write $L = \graph(l |_{\del \Omega})$.  Trivially $E$ has locally-finite perimeter in $\R^{n+1} \setminus L$, since $\del E$ is locally Lipschitz in this region, and moreover $\del [E] \llcorner (\Omega \times \R) = T$, $\del [E] \llcorner (\R^{n+1} \setminus (\Omega \times \R)) = [G]$.  Since both $u$ and $l$ are continuous, we have
\begin{equation}
||\del[E]||((\del\Omega \times \R) \setminus L) = 0. \label{eqn:E-L}
\end{equation}

Fix $r > 0$, $\eps > 0$, and let $\Omega_\eps = \{ x \in \Omega : d(x, \del \Omega) > \eps \}$.  Then $D_\eps := (\Omega_\eps \times \R) \cap B_r(0)$ is a convex set of locally-finite perimeter, and
\begin{equation}\label{eqn:Deps-bound}
||\del [D_\eps]||(\overline{B_r(0)}) \leq \haus^n(\del B_r(0)) \leq \omega_n r^n.
\end{equation}
Taking $S = \del [E \cup D_\eps] - \del [E]$ in \eqref{eqn:graph-minz}, we deduce $||T||(D_\eps) \leq c(n) r^n$, and hence taking $\eps \to 0$ gives
\begin{equation}\label{eqn:D-bound}
||T||( B_r(0) \cap (\Omega \times \R)) = || \del [E]||( B_r(0) \cap (\Omega \times \R)) \leq c(n) r^n.
\end{equation}
On the other hand, if $K$ is any bound for the Lipschitz constant of $l|_{B_r(0)}$, then we have the trivial bound
\begin{equation}\label{eqn:G-bound}
||\del [E]||( B_r(0) \setminus (\overline{\Omega} \times \R)) \leq c(n, K) r^n.
\end{equation}
Combining \eqref{eqn:E-L}, \eqref{eqn:D-bound}, \eqref{eqn:G-bound}, we obtain
\begin{equation}\label{eqn:E-minus-L}
||\del [E]||(B_r(0) \setminus L) \leq c(n, K) r^n.
\end{equation}

Now take $X \in C^1_c(B_r(0), \R^{n+1})$, and fix $0 < \tau < r$.  Let $d : B_{2r}(0) \to \R$ be a smooth function such that 
\[
d(x) \in [\tau/2, \tau] \implies x \in B_{2\tau}(L) \setminus B_{\tau/4}(L), \quad |Dd| \leq 2.
\]
(For example, $d$ could be a mollification of $d(\cdot, L)$.)  Let $\eta(t) : \R \to \R$ be a non-negative, increasing function which is $\equiv 0$ on $(-\infty, 1/2]$, $\equiv 1$ on $[1, \infty)$, and $|\eta'| \leq 10$.  Set $\phi = \eta(d/\tau)$.

Since $L$ is countably $(n-1)$-rectifiable with locally-finite $\haus^{n-1}$-measure (being the graph of a locally Lipschitz function over a locally Lipschitz domain), we have (\cite[Section 3.2.29]{Fed})
\begin{equation}\label{eqn:minkowski}
\limsup_{\tau \to 0} \tau^{-2} \leb^{n+1}(B_\tau(L) \cap B_r(0)) = \haus^{n-1}(L \cap B_r(0)) \leq c(n, K) r^{n-1} .
\end{equation}

We compute, using \eqref{eqn:E-minus-L}, \eqref{eqn:minkowski}, and taking $\tau > 0$ sufficiently small:
\begin{align*}
\left| \int_E \phi \mathrm{div}(X) d\leb^{n+1} \right|
&\leq \left| \int \phi X \cdot \nu_E d|| \del[E]|| \right| + |X|_{C^0} \left| \int_{B_r(0)} |D\phi| d\leb^{n+1} \right| \\
&\leq |X|_{C^0} ||\del[E]||(B_r(0) \setminus L) + 20 |X|_{C^0} \tau^{-1} \leb^{n+1}(B_r(0) \cap B_{2\tau}(L)) \\
&\leq c(n, K) |X|_{C^0} r^n + c(n, K) |X|_{C^0} \tau r^{n-1} .
\end{align*}
Here $\nu_E$ is the outward unit normal of $\del [E]$ away from $L$.  Letting $\tau \to 0$ gives that
\[
\left| \int_E \mathrm{div}(X) d\leb^{n+1} \right| \leq c(n, K)|X|_{C^0} r^n,
\]
which implies that $E$ has locally-finite perimeter, and hence $\del [E]$ is integral.  As a corollary, since $\haus^n(L) = 0$ we get 
\begin{equation}\label{eqn:E-no-L}
||\del[E]||(L) = 0.
\end{equation}

Finally, we observe that \eqref{eqn:E-L}, \eqref{eqn:E-no-L} imply
\[
\del[E] = \del [E] \llcorner (\Omega \times \R) + \del[E] \llcorner ((\R^{n+1} \setminus \overline{\Omega}) \times \R).
\]
which gives the equalities of Item a).

Item a) implies $T = \del[E] - [G] \in \cI_n(\R^{n+1})$.  We show $T$ is mass-minimizing in $\R^{n+1}$.  Let $p_\eps : \R^{n+1} \to \overline{\Omega_\eps} \times \R$ denote the nearest-point projection.  By Lemma \ref{lem:convex} $p_\eps$ is area-decreasing, and by construction $(p_0)_\sharp T = T$.

Take $r > 0$ and $S \in \cI_n(\R^{n+1})$ satisfying $\del S = 0$, $\spt S \subset B_r(0)$.  From \eqref{eqn:graph-minz} we have for every $\eps > 0$:
\begin{align*}
||T||(B_r(0))
&\leq ||T + (p_\eps)_\sharp S ||(B_R(0)) \\
&\leq ||T + S||(B_r(0) \cap (\Omega_\eps \times \R)) \\
&\quad + ||T||( B_r(0)  \setminus (\Omega_\eps \times \R)) + ||S||( B_r(0) \setminus (\Omega_\eps \times \R)).
\end{align*}
Taking $\eps \to 0$ and using that $||T||(\R^{n+1} \setminus (\Omega \times \R)) = 0$ gives
\[
||T||(B_r(0)) \leq ||T+S||(B_r(0) \cap \Omega \times \R) + ||S||(B_r(0) \setminus \Omega \times \R) = ||T+S||(B_r(0)),
\]
which is the required inequality.  This proves Item b).  Lastly, Item c) follows from \eqref{eqn:E-minus-L}, \eqref{eqn:E-no-L}.
\end{proof}

\section{Proof of Main Theorems}

In view of Lemma \ref{lem:graph-minz}, Theorem \ref{thm:main1} trivially follows from Theorem \ref{thm:main2}.  We will prove Theorem \ref{thm:main2} differently depending on whether $\Omega$ is a half-space (Case 1), or a slab (Case 2), or neither of the two (Case 3).  The first case is effectively handled by Lemma \ref{lem:plane-bern}, the second case by Lemma \ref{lem:slab-bern}, and the third by Lemma \ref{lem:convex-bern}.  Before proving these lemmas and the main theorem, we make a few observations that will come in useful.

First, given $T$ as in Theorem \ref{thm:main2}, and any $B_r(x) \subset \R^{n+1}$, then by taking $S = \del [E \cup B_r(x)] - \del[E]$ as a comparison current we get the bound
\begin{equation}\label{eqn:T-bound}
||T||(B_r(x)) \leq \haus^n(\del B_r(x)) + \haus^n(G \cap B_r(x)) \leq c(n) r^n.
\end{equation}

Second, if $T$ as in Theorem \ref{thm:main2} has $\del T = [\R^{n-1}]$, then if we denote by $|T|$ the varifold induced by $T$ and define $\tau(x_1, \ldots, x_{n+1}) = (x_1, \ldots, x_{n-1}, -x_n, -x_{n+1})$, then the varifold
\begin{equation}\label{eqn:T-refl}
V = |T| + \tau_\sharp |T|
\end{equation}
is a stationary integral varifold in $\R^{n+1}$ with $||V||(B_r(0)) = 2 ||T||(B_r(0))$ (c.f. \cite{Allard1975}).

Third, by standard monotonicity for stationary varifolds (\cite{Sim}), for the same $T$ we have the monotonicity:
\begin{equation}\label{eqn:T-mono}
\theta_T(0, r) - \theta_T(0, s) = \frac{1}{\omega_n} \int_{B_r(0) \setminus B_s(0)} \frac{|x^\perp|^2}{|x|^{n+2}} d||T||(x),
\end{equation}
where $\theta_T(x, r) = \omega_n^{-1} r^{-n} ||T||(B_r(x))$ is the usual Euclidean density ratio, and $x^\perp$ is the orthogonal projection of $x$ onto $T_x T^\perp$, defined for $||T||$-a.e. $x$.  In particular, $\theta_T(0, r)$ is increasing, and constant precisely when $T$ is a cone, i.e. when $(\eta_{0, \lambda})_\sharp T = T$ for all $\lambda > 0$.  Write $\theta_T(x) = \lim_{r \to 0} \theta_T(x, r)$.

\vspace{3mm}

To handle Case 3 we can use a straightforward barrier argument, more or less just rotating a plane until it touches $\spt T$, and then using the strong maximum principle for stationary varifolds \cite{SW}. \footnote{In fact since our barriers are planes, and $T$ is a minimizing boundary near the point of contact, one could instead use Allard's regularity theorem \cite{All} and the strong maximum principle for 2nd order elliptic PDEs.}  We get very strong rigidity because our assumption on $\Omega$ precludes the possibility that $\overline{\Omega} + \R v$ contains our barrier half-hyperplane, so $T$ cannot have any non-zero pieces outside of $\overline{\Omega}$.
\begin{lemma}\label{lem:convex-bern}
Let $\Omega$ be an open convex cone in $\R^n$, which is a proper subset of some half-space, and let $v \in \R^{n+1}$ with $v \cdot e_{n+1} > 0$.  Suppose $T \in \cI_n(\R^{n+1} \cap \{ x_{n+1} > 0 \})$ is mass minimizing in $\{ x_{n+1} > 0\}$, and has the form $T = \del [E] \llcorner \{ x_{n+1} > 0 \}$ for some set $E \subset \overline{\Omega} + \R v$. Then $T = 0$.
\end{lemma}

\begin{proof}
After a rotation/translation in $\R^n$ we can arrange so that if $K$ is the space of translational symmetry of $\Omega$ (i.e. so that $v \in K \iff \Omega + v =\Omega$), then $\Omega \subset \R^{n-1} \times (0, \infty)$ and $\overline{\Omega} \cap \R^{n-1} = K$.  Note that if $n = 2$ then $K = \{0\}$, and in general $K$ is at most $(n-2)$-dimensional.  Suppose $T \neq 0$.  Then
\[
\alpha = \sup_{x \in \spt T \setminus \R^{n-1}} \frac{d(x, \R^n)}{d(x, \R^{n-1})} > 0.
\]
Choose a sequence $x_i$ so that
\[
\frac{d(x_i, \R^n)}{d(x_i, \R^{n-1})} \to \alpha .
\]
After translating $T$ by an element of $K$, we can assume that $x_i \in K^\perp$.

Let $T_i = (\eta_{0, |x_i|})_\sharp T$, and $y_i = x_i/|x_i|$.  Since $T_i$ is a minimizing boundary in $\{x_{n+1} > 0\}$, we have bounds of the form $||T_i||(U) \leq c(U)$ for any $U \subset\subset \{ x_{n+1} > 0\}$, and with $c$ independent of $i$.  Therefore, after passing to a subsequence, we can find a $T' \in \cI_n(\{x_{n+1} > 0\})$ which is minimizing and has zero boundary, so that $T_i \to T'$ in $\{ x_{n+1} > 0 \}$ as both currents and measures.  Upper-semi-continuity of density implies $\spt T_i \to \spt T'$ in the local Hausdorff distance in $\R^{n+1} \setminus \R^n$, and $\spt T' \subset \overline{\Omega} + \R v$.

Passing to a further subsequence, we can assume $y_i \to y \in \del B_1 \cap K^\perp \cap (\overline{\Omega} + \R v)$.  Since $y \in K^\perp$, we have $d(y, \R^{n-1}) > 0$, and hence $d(y, \R^n) > 0$ also.  In particular $y \in \spt T'$.

By construction, and by the local Hausdorff convergence of supports, we have
\[
\alpha = \frac{d(y, \R^n)}{d(y, \R^{n-1})} = \sup_{z \in \spt T' \setminus \R^{n-1}} \frac{d(z, \R^n)}{ d(z, \R^{n-1})}.
\]
Therefore if $H$ is the hyperplane containing both $\R^{n-1}$ and $y$, we deduce that $\spt T'$ lies to one side of $H$, and touches $H$ at $y$.  By the strong maximum principle \cite{SW}, we deduce that $H \cap \{ x_{n+1} > 0 \} \subset \spt T' \cap \{ x_{n+1} > 0 \} \subset (\overline{\Omega} + \R v) \cap \{ x_{n+1} > 0\}$.  This is a contradiction, since $v \cdot e_{n+1} \neq 0$ and $\overline{\Omega} \neq \R^{n-1} \times [0, \infty)$.
\end{proof}

For Case 2, we use in part a similar barrier argument, but we additionally have to use very strongly the structure $T = \del [E]$, for $E$ contained in an $(n+1)$-dimensional slab.  This is to rule out $T$ looking like a vertical plane very far away from its boundary.  Of course if $E$ is not contained in a slab, this could occur, but when $E$ is sandwiched as well as $T$ this violates the mass-minimizing condition.
\begin{lemma}\label{lem:slab-bern}
Let $\Omega = \R^{n-1} \times (0, 1)$, and $v \in \R^{n+1}$ with $v \cdot e_{n+1} > 0$.  Suppose $T \in \cI_n(\R^{n+1} \cap \{ x_{n+1} > 0 \})$ is mass-minimizing in $\{ x_{n+1} > 0\}$, and has the form $T = \del [E] \llcorner \{ x_{n+1} > 0\}$ for some set $E \subset \overline{\Omega} + \R v$.  Then $T = 0$.
\end{lemma}

\begin{proof}
Assume that $T \neq 0$.  Then
\[
\alpha = \sup_{x \in \spt T} d(x, \R^n) > 0.
\]
Choose a sequence $x_i \in \spt T$ with $d(x_i, \R^n) \to \alpha$.  After translating in the $\R^{n-1}$ direction we can assume $x_i = z_i + v r_i$ for $z_i \in \{ 0^{n-1}\} \times [0, 1]$ and $r_i \in (0, \infty)$.  We can also assume $z_i \to z \in [0, 1]$.

Let $T_i = (\eta_{x_i, 1})_\sharp T \equiv \del [ \eta_{x_i, 1}(E)] \llcorner \{ x_{n+1} > -d(x_i, \R^n)\}$.  Since each $T_i$ is a mass-minimizing boundary in $\{ x_{n+1} > - d(x_i, \R^n)\}$, for any $U \subset \subset \{ x_{n+1} > -\alpha\}$ we have bounds of the form $||T_i||(U) \leq c(U)$ for $c$ independent of $i$.  Passing to a subsequence, we can find a mass-minimizing $T' \in \cI_n( \{ x_{n+1} > -\alpha \})$ so that $T_i \to T'$ as both currents and measures in $\{ x_{n+1} > -\alpha\}$.  Since $0 \in \spt T_i$, $0 \in \spt T'$ by upper-semi-continuity of density.  Moreover, by taking the current limit of $[\eta_{x_i, 1}(E)]$ as well, $T'$ has the form $T' = \del[E'] \llcorner \{ x_{n+1} > -\alpha\}$ for $E' \subset (\overline{\Omega} - z) + \R v$.

If $\alpha < \infty$, then $\spt T'$ lies below the hyperplane $H = \{x_{n+1} = 0\}$, and touches $H$ at $0$, and $T'$ is mass-minimizing in a neighborhood of $H$ (since $\alpha > 0$).  The strict maximum principle implies $H \subset \spt T' \subset (\overline{\Omega} - z) + \R v$, which is a contradiction.

If $\alpha = \infty$, then $T' = \del [E']$ is a non-zero mass-minimizing boundary in $\R^{n+1}$.  Therefore we can find a sequence $r_i \to \infty$ so that the rescaled $T_i' = (\eta_{0, r_i})_\sharp T' \equiv \del [\eta_{0, r_i}(E')]$ converge as both currents and measures to a mass-minimizing boundary $T'' = \del[E''] \in \cI_n(\R^{n+1})$.  Since $0 \in \spt T'_i$, $0 \in \spt T''$ also.  But now $E'' \subset \R^{n-1} + \R v$ and so as $(n+1)$-currents $[E''] = 0$.  This is a contradiction.
\end{proof}

Case 1 follows from the below characterization of minimizing hypercones with linear bounary.  This fact is essentially a direct consequence of \cite{HS79}, but since our setting is much simpler, one can give a much easier proof.
\begin{lemma}\label{lem:plane-bern}
Let $\Omega = \R^{n-1} \times (0, \infty) \subset \R^n$, and $a \in \R$.  Suppose $T \in \cI_n(\R^{n+1})$ is a mass-minimizing cone in $\R^{n+1}$ of the form $T = \del [E] - [G]$, where $E \subset \R^{n+1}$ satisfies $E \cap \{ x_n < 0 \} = \{ x_{n+1} < a x_n , x_n < 0 \}$, and $[G] = [\{ x_{n+1} = a x_n, x_n < 0 \}]$ oriented with the upwards unit normal.  Then $T = [H]$ for some half-hyperplane $H$, and in particular $\theta_T(0) = 1/2$.
\end{lemma}

\begin{proof}
We induct on $n$.  If $n = 1$, then $T = \sum_{i=1}^N [l_i]$ is a finite union of rays (with orientation) from the origin contained in $[0, \infty) \times \R$.  Since $\del T = [\{0\}]$ and $T$ is mass-minimizing, a standard comparison argument implies either $N = 1$ or each $l_i \subset \{0\} \times \R$.  If we had $\spt T = \{ 0\} \times \R$, then we would have $\spt \del [E] = \{ x_{2} = ax_1, x_1 < 0 \} \cup \{ x_1 = 0 \}$, which is impossible.  So we must have $N = 1$, which proves the $n = 1$ case.

Take $n \geq 2$, and suppose by induction Lemma \ref{lem:plane-bern} holds for any $n' \leq n-1$ in place of $n$.  We claim that $T$ is regular and multiplicity-one in a neighborhood of $\R^{n-1} \setminus \{0\}$.  To see this, take any $x \in \R^{n-1} \setminus \{0\}$, and then by \eqref{eqn:T-bound}, \eqref{eqn:T-mono}, and compactness for mass-minimizing currents we can find a mass-minimizing tangent cone $T' \in \cI_n(\R^{n+1})$ to $T$ at $x$, satisfying $\del T' = [\R^{n-1}]$, $\theta_{T'}(0) = \theta_T(x)$, and $T' = \del[E'] - [G]$.

Since $x \neq 0$, \eqref{eqn:T-mono} implies we can write $T' = [\R x] \times T''$ for $T''$ a mass-minimizing cone in $(\R x)^\perp \cong \R^n$ satisfying the hypotheses of Lemma \ref{lem:plane-bern} with $n-1$ in place of $n$.  By induction we have $\theta_{T''}(0) = 1/2$, and hence $\theta_T(x) = 1/2$ for every $x \in \R^{n-1} \setminus \{0\}$.  By reflecting $T$ about $\R^{n-1}$ as in \eqref{eqn:T-refl}, we obtain a stationary integral varifold $V$ in $\R^{n+1}$, satisfying $\theta_V(x, 0) = 1$ for every $x \in \R^{n-1} \setminus \{0\}$.  Allard's theorem \cite{All} implies $V$ is regular near $\R^{n-1} \setminus \{0\}$, and hence $T$ is regular near $\R^{n-1} \setminus \{0\}$ also.  This proves our claim.

Now set
\begin{equation}\label{eqn:theta}
\theta = \sup_{x \in \spt T \setminus \R^{n-1}} \arcsin\left( \frac{x \cdot e_{n+1}}{d(x, \R^{n-1})}\right),
\end{equation}
where $\arcsin : [-1, 1] \to [-\pi/2, \pi/2]$.  If $\theta = -\pi/2$, then $\spt T \subset H = \{ x_{n+1} < 0, x_n = 0 \}$, and so by our hypothesis on $T$ and $E$ we have $T = [H]$.

Suppose $\theta > -\pi/2$, and let $H$ be the half-plane $\R^{n-1} \times \{ t \cos \theta e_n + t \sin \theta e_{n+1} : t \in (0, \infty) \}$. By our choice of $\theta$ and by our assumption on $T$, $\spt T$ lies inside the region bounded by $H$ and $\{ x_{n+1} \leq 0, x_n = 0 \}$.  Choose a sequence $x_i \in \spt T \cap \del B_1 \setminus \R^{n-1}$ maximimizing \eqref{eqn:theta} (recalling $\spt T$ is dilation-invariant), and WLOG assume $x_i \to x \in \spt T \cap \del B_1 \cap H$.  If $x \not \in \R^{n-1}$ the strong maximum principle \cite{SW} implies $H \subset \spt T$.  Otherwise, if $x \in \R^{n-1}$, then since $T$ is regular near $\R^{n-1} \setminus \{0\}$ the Hopf lemma implies $H = \spt T$ near $x$, and hence $H \subset \spt T$ also.

Endow $[H]$ with the correct orientation so that $\del [H] = -\del [G] = \del T$.  Since $T$ is integral and multiplicity-one, from the standard monotonicity formula we can write $T = [H] + T'$, $||T|| = ||[H]|| + ||T'||$, where $T'$ is a mass-minimizing cone in $\R^{n+1} \setminus \R^{n-1}$ with $\del T' = 0$.  Since $T$ is regular (and multiplicity-one) near $\R^{n-1} \setminus \{0\}$, we have $T = [H]$ in a neighborhood $U$ of $\R^{n-1} \setminus \{0\}$.  We deduce that the varifold $|T'|$ is a stationary integral varifold cone in $\R^{n+1}$ which is supported in a half-space $\R^{n-1} \times [0, \infty) \times \R$, and hence either $T' = 0$ or $\spt T' = \R^{n-1} \times \{0\} \times \R$.  But $||T'||(U) = 0$, and so the latter is impossible.  We deduce that $T = [H]$.
\end{proof}

\begin{proof}[Proof of Theorem \ref{thm:main2}]
First assume that $\Omega$ is a half-space.  Then after a rotation, translation in $\R^n$, we can assume $\Omega = \R^{n-1} \times (0, \infty)$.  After a further rotation in $\R^{n-1} \times \{0\} \times \R$, we can assume that $l(x) = a x_n$ for some $a \in \R$, and $E \cap \{ x_n < 0 \} = \{ x_{n+1} < a x_n, x_n < 0 \}$, and (hence) $T = \del [E] - [\{ x_{n+1} = ax_n, x_n < 0\}]$.

Choose any sequence of radii $r_i \to 0$, and let $T_i' = (\eta_{0, r_i})_\sharp T$.  By \eqref{eqn:T-bound}, \eqref{eqn:T-mono}, and compactness of mass-minimizing integral currents, we can pass to a subsequence and get convergence $T_i' \to T'$ as both currents and measures, where $T' \in \cI_n(\R^{n+1})$ is a mass-minimizing cone.  Moreover, we continue to have $T' = \del [E'] - [G]$, for some open set $E'$ satisfying $E' \cap \{ x_n < 0 \} = E \cap \{ x_n < 0 \}$.  Lemma \ref{lem:plane-bern} implies that $\theta_{T'}(0) = 1/2$, and hence $\theta_T(0) = 1/2$ also.

We can apply the same reasoning to a sequence $r_i \to \infty$ to deduce $\theta_T(0, \infty) = 1/2$ also, and therefore $T$ is a cone.  We get that $T = T' = [H]$ for some half-hyperspace $H$.  This proves the first case.

Assume now $\Omega$ is the slab trapped between two hyperplanes.  Let $R$ be a composition of rotations, dilations, and translations of $\R^{n+1}$ which takes the subgraph $\{ x_{n+1} < l(x') : x' \in \R^n \}$ to $\{ x_{n+1} < 0\}$, and takes $\graph(l|_{\Omega})$ to the slab $\Omega' = \R^{n-1} \times (0, 1)$.  Let $v = R(e_{n+1})$, $T' = R_\sharp T$, $E' = R(E)$.  Then $v \cdot e_{n+1} > 0$, $R(\Omega \times \R) = \Omega' + \R v$, $E' \setminus (\overline{\Omega'} + \R v) = \{ x_{n+1} < 0 \} \setminus ( \overline{\Omega'} + \R v)$, and $T' = \del[E'] - [\R^{n-1} \times ((-\infty, 0) \cup (1, \infty))]$.

Note that $T' \llcorner \{ x_{n+1} > 0 \} = \del [E'] \llcorner \{ x_{n+1} > 0 \}$ and $E' \cap \{ x_{n+1} > 0\} \subset (\overline{\Omega'} + \R v)$.  Therefore we can apply Lemma \ref{lem:slab-bern} to deduce $T' \llcorner \{ x_{n+1} > 0 \} = 0$.  By replacing $E'$ with $\R^{n+1} \setminus E'$ and reflecting around $\R^n$, we can again apply Lemma \ref{lem:slab-bern} to get $T' \llcorner \{ x_{n+1} < 0 \} = 0$.  So $\spt T' \subset \R^n$, and so we must have $E' = \{ x_{n+1} < 0\}$.  This proves the second case.

Assume finally that $\Omega$ is neither a half-space or a slab.  Then there must be two points in $\del \Omega$ with non-parallel supporting hyperplanes.  It follows that there is an open convex cone $\Omega'$, which is not a half-space, so that after a translation in $\R^n$ we have $\Omega \subset \Omega'$.

We can apply a similar rotation/translation/dilation $R$ as in Case 2, to rotate $\graph(l|_{\Omega'})$ to some convex cone $\Omega'' \subset \R^n$ (which is again is not a half-space).  If $T'' = R_\sharp T$, then arguing as in Case 2 except with Lemma \ref{lem:convex-bern} in place of Lemma \ref{lem:slab-bern}, we get that $\spt T'' \subset \R^n$, and hence $T = [\graph(l|_\Omega)]$.  This completes Case 3, and the proof of Theorem \ref{thm:main2}.
\end{proof}

\section*{Acknowledgements}
We thank Jintian Zhu for pointing out an error in our original draft, and to the referee for several comments and suggetions that greatly improved this manuscript.  Z. W. is grateful to the China Scholarship Council for supporting his visiting study at University of Notre Dame, and the Department of Mathematics at University of Notre Dame for its hospitality.

\end{document}